\documentclass[11pt]{amsart}
\usepackage{amscd,amssymb}
\usepackage{amsthm,amsmath,amssymb}
\usepackage[matrix,arrow]{xy}

\theoremstyle{definition}
\newtheorem{example}[equation]{Example}
\newtheorem{definition}[equation]{Definition}
\newtheorem{theorem}[equation]{Theorem}
\newtheorem{lemma}[equation]{Lemma}
\newtheorem{corollary}[equation]{Corollary}
\newtheorem{conjecture}[equation]{Conjecture}

\newtheorem*{question*}{Question}

\newtheorem*{problem*}{Problem}

\theoremstyle{remark}
\newtheorem{remark}[equation]{Remark}

\makeatletter\@addtoreset{equation}{section} \makeatother

\def \Q {\mathbb{Q}}
\def \Z {\mathbb{Z}}
\def \A {\mathrm{A}}

\def \F {\mathbb{F}}
\def \C {\mathbb{C}}

\def \D {\mathrm{D}}
\def \SS {\mathrm{S}}
\def \PSp {\mathrm{PSp}}
\def \SL {\mathrm{SL}}
\def \GL {\mathrm{GL}}
\def \PSL {\mathrm{PSL}}
\def \PGL {\mathrm{PGL}}
\def \SU {\mathrm{SU}}
\def \PSU {\mathrm{PSU}}
\def \HaJ {\mathrm{HaJ}}
\def \Sp {\mathrm{Sp}}
\def \qlin {\sim_{\mathbb{Q}}}

\def \le {\leqslant}
\def \ge {\geqslant}

\author{Ivan Cheltsov and Constantin Shramov}

\title{Six-dimensional exceptional quotient singularities}

\thanks{The authors were partially supported by AG Laboratory GU-HSE, RF
government grant 11~11.G34.31.0023.
The~first author was supported by the~grants NSF DMS-0701465 and
EPSRC EP/E048412/1, the~second~author was supported by the~grants
RFFI~08-01-00395-a, RFFI~11-01-00185-a, RFFI~11-01-00336-a,
N.Sh.-1987.2008.1 and N.Sh.-4713.2010.1.}

\address{University of Edinburgh, Edinburgh EH9 3JZ, UK, \texttt{I.Cheltsov@ed.ac.uk}}

\address{Steklov Institute of Mathematics, Moscow 119991, Russia, \texttt{shramov@mccme.ru}}

\address{Laboratory of Algebraic Geometry, GU-HSE, 7 Vavilova street, Moscow 117312, Russia}%

\begin{document}

\begin{abstract}
We classify six-dimensional exceptional quotient singularities and
show that seven-dimensional exceptional quotient singularities do
not exist. Inter alia we prove that the irreducible
six-dimensional projective representation of the sporadic simple
Hall--Janko group gives rise to an exceptional quotient
singularity.
\end{abstract}

\maketitle


We assume that all varieties are projective, normal, and defined
over $\mathbb{C}$.

\section{Introduction}
\label{section:intro}

Let $(V\ni O)$ be a~germ of a Kawamata log terminal singularity
(see \cite[Definition~3.5]{Ko97}), and let $\xi\colon \bar{V}\to
V$ be a~resolution of singularities of the~variety $V$. Then
$$
K_{\bar{V}}\sim_{\mathbb{Q}}\xi^{*}\big(K_{V}\big)+\sum_{i=1}^{r}b_{i}E_{i},
$$
where $E_{i}$ is a~$\xi$-exceptional divisor, and
$b_{i}\in\mathbb{Q}$. Let $B$ be an effective $\mathbb{Q}$-divisor
on $V$. Put
$$
B=\sum_{i=1}^{m}a_{i}B_{i},
$$
where $B_{i}$ is a~prime Weil divisor on $V$, and
$a_{i}\in\mathbb{Q}_{\geqslant 0}$. Suppose that $B$ is
a~$\mathbb{Q}$-Cartier divisor.~Then
$$
\sum_{i=1}^{m}a_{i}\bar{B}_{i}\sim_{\mathbb{Q}}\xi^{*}\Bigg(\sum_{i=1}^{m}a_{i}B_{i}\Bigg)-\sum_{i=1}^{r}c_{i}E_{i},
$$
where $\bar{B}_{i}$ is the~proper transform of the~divisor $B_{i}$
on the~variety $\bar{V}$. Suppose that
$$
\left(\bigcup_{i=1}^{m}\bar{B}_{i}\right)\bigcup\left(\bigcup_{i=1}^{r}E_{i}\right)
$$
is a~divisor with simple normal crossing.

\begin{definition}[{\cite[Definition~8.1]{Ko97}}]
\label{definition:lct} The~log canonical threshold of the~divisor
$B$ at $O$ is
$$
\mathrm{c}_{O}\big(X,B\big)=\mathrm{min}\Bigg(\mathrm{min}\Bigg\{\frac{1}{a_{i}}\Bigg\vert\ O\in B_{i}\Bigg\}, \mathrm{min}\Bigg\{\frac{b_{i}+1}{c_{i}}\Bigg\vert\ O\in\xi\big(E_{i}\big)\Bigg\}\Bigg)\in\mathbb{Q}_{>0}\cup\big\{+\infty\big\}%
$$
\end{definition}

\begin{definition}[{\cite[Definition~2.3.1]{ChiYau08}}]
\label{definition:lct-mult} The~log canonical multiplicity of
the~divisor $B$ at $O$ is
$$
\mathrm{\mu}_{O}\big(X,B\big)=\mathrm{max}\left\{\alpha+\beta\ \left|\aligned%
&O\in\xi\Bigg(\Big(\bar{B}_{i_{1}}\cap\ldots\cap \bar{B}_{i_{\alpha}}\Big)\bigcap\Big(E_{k_{1}}\cap\ldots\cap E_{k_{\beta}}\Big)\Bigg)\ \text{and}\\
&\frac{1}{a_{i_{1}}}=\ldots=\frac{1}{a_{i_{\alpha}}}=\frac{b_{k_{1}}+1}{c_{k_{1}}}=\ldots=\frac{b_{k_{\beta}}+1}{c_{k_{\beta}}}=\mathrm{c}_{O}\big(X,B\big)\\
\endaligned\right.\right\}\in\mathbb{Z}_{\geqslant 0},%
$$
\end{definition}

One can show that t
he~numbers $\mathrm{c}_{O}(X,B)$ and $\mathrm{\mu}_{O}(X,B)$ do
not depend on the~choice of the~log resolution $\xi$.

\begin{definition}[{\cite[Definition~2.5]{MarPr99}}]
\label{definition:exceptional} The~singularity $(V\ni O)$
is~exceptional if
$$
\mathrm{\mu}_{O}\big(X,B\big)\leqslant 1
$$
for every effective $\mathbb{Q}$-Cartier $\mathbb{Q}$-divisor $B$
on the~variety $V$.
\end{definition}

Note that Definition~\ref{definition:exceptional} looks different
from \cite[Definition~2.5]{MarPr99}, but they are equivalent.
One can show that exceptional singularities exist in any dimension greater than
$1$ (see~\cite[Example~3.13]{ChSh09}).

\begin{example}
\label{example:DuVal} Suppose that $\mathrm{dim}(V)=2$ and
$-K_{V}$ is Cartier. Then the~singularity $(V\ni O)$ is exceptional
if and only if it is a~Du Val singularity of type~$\mathbb{E}_{6}$,
$\mathbb{E}_{7}$ or~$\mathbb{E}_{8}$.
\end{example}

Let $G$ be a~finite subgroup in $\GL_{n+1}(\mathbb{C})$, where
$n\geqslant 1$. Put
$$
\bar{G}=\phi\big(G\big)\subset\mathrm{Aut}\big(\mathbb{P}^{n}\big)\cong
\PGL_{n+1}\Big(\mathbb{C}\Big),
$$
where
$\phi\colon\GL_{n+1}(\mathbb{C})\to\mathrm{Aut}(\mathbb{P}^{n})$ is
the~natural projection. Put
$$
\mathrm{lct}\Big(\mathbb{P}^{n},\bar{G}\Big)=\mathrm{sup}\left\{\lambda\in\mathbb{Q}\ \left|%
\aligned
&\text{the~log pair}\ \left(\mathbb{P}^{n}, \lambda D\right)\ \text{has log canonical singularities}\\
&\text{for every $\bar{G}$-invariant effective $\mathbb{Q}$-divisor}\ D\sim_{\mathbb{Q}} -K_{\mathbb{P}^{n}}\\
\endaligned\right.\right\}\in\mathbb{R}.%
$$

\begin{remark}[{cf. Appendix~\ref{section:alpha}}]
\label{remark:alpha} It follows from
\cite[Theorem~A.3]{ChSh08c}~that
$$
\mathrm{lct}\Big(\mathbb{P}^{n},\bar{G}\Big)=\alpha_{\bar{G}}\big(\mathbb{P}^{n}\big),
$$
where $\alpha_{\bar{G}}(\mathbb{P}^{n})$ is
the~$\bar{G}$-invariant $\alpha$-invariant introduced in
\cite{Ti87} and \cite{TiYa87}.
\end{remark}

We are going to study the~quotient singularity
$\mathbb{C}^{n+1}\slash G$.

\begin{remark}
\label{remark:reflections} Let $R\subseteq G$ be a~subgroup
generated by all reflections in $G$ (see~\cite[\S 4.1]{Spr77}).
Then
the~quotient $\mathbb{C}^{n+1}\slash R$ is isomorphic
to $\mathbb{C}^{n+1}$ (see~\cite[Theorem~4.2.5]{Spr77}).
Moreover, the~subgroup $R\subseteq G$ is
normal, and the singularity $\C^{n+1}/G$ is isomorphic to the
singularity $\C^{n+1}/(G/R)$.
Note that the~subgroup $R$ is
trivial if $G\subset\SL_{n+1}(\mathbb{C})$.
If $G=R$ (in particular, if~$G$ is a~trivial group),
then the singularity $\C^{n+1}/G\cong\C^{n+1}$
is not exceptional.
\end{remark}

The following result provides a characterization of exceptional
quotient singularities.

\begin{theorem}[{\cite[Theorem~3.17]{ChSh09}}]
\label{theorem:criterion} Let $G\subset\GL_{n+1}(\C)$ be a finite
subgroup that does not contain reflections. Then
$\mathbb{C}^{n+1}\slash G$ is exceptional if and only if for any
$\bar{G}$-invariant effective $\mathbb{Q}$-divisor $D$ on
$\mathbb{P}^{n}$ such that $D\sim_{\mathbb{Q}}
-K_{\mathbb{P}^{n}}$, the~log pair $(\mathbb{P}^{n}, D)$ is
Kawamata log terminal.
\end{theorem}

\begin{corollary}
\label{corollary:criterion} Let $G\subset\GL_{n+1}(\C)$ be a finite
subgroup that does not contain reflections.  Then
the singularity~$\mathbb{C}^{n+1}\slash G$ is exceptional if
$\mathrm{lct}(\mathbb{P}^{n},\bar{G})>1$. Moreover,
the~singularity $\mathbb{C}^{n+1}\slash G$ is not exceptional if
either $\mathrm{lct}(\mathbb{P}^{n},\bar{G})<1$, or $G$ has
a~semi-invariant of degree~at~most~$n+1$.
\end{corollary}

Note that the~assumption that $G$ contains no reflections is
crucial for Theorem~\ref{theorem:criterion}.

\begin{example}
\label{example:reflections} Let $G\subset\GL_4(\C)$ be
the~subgroup $\# 32$ in \cite[Table~VII]{SheTo54}. Then $G$ is
generated by reflections (see \cite{SheTo54}). Thus,
the~singularity $\mathbb{C}^{4}\slash G$ is not exceptional by
Remark~\ref{remark:reflections}. On the other hand, it follows
from \cite[Theorem~4.13]{ChSh09} that
$\mathrm{lct}(\mathbb{P}^{3},\bar{G})\geqslant 5/4$.
One can produce similar examples for two-dimensional and three-dimensional
singularities.
\end{example}

\begin{definition}[see \cite{Bli17}]
\label{definition:primitive} The~subgroup $G\subset\GL_{n+1}(\C)$ is
primitive if there is~no~non-trivial decomposition
$$
\mathbb{C}^{n+1}=\bigoplus_{i=1}^{r}V_{i}
$$
such that
for any $g\in G$ and any $i$ there is some $j=j(g)$ such that
$g(V_{i})=V_{j}$.
The~subgroup $\bar{G}\subset\PGL_{n+1}(\C)$
is primitive if $G$ is primitive.
\end{definition}

Up to conjugation, there are finitely many primitive finite
subgroups in $\SL_{n+1}(\mathbb{C})$~(see~\cite{Col07}).

\begin{theorem}[{see~\cite[Proposition~2.1]{Pr00}
or~\cite[Corollary~3.20]{ChSh09}}]
\label{theorem:primitive} Let $G\subset\GL_{n+1}(\C)$ be a finite
subgroup that does not contain reflections. If
$\mathbb{C}^{n+1}\slash G$ is exceptional, then $G$ is primitive.
\end{theorem}

Exceptional quotient singularities of dimension up to $5$ are
classified in \cite{Sho93}, \cite{MarPr99}, \cite{ChSh09}.

\begin{theorem}[{\cite{Sho93}, \cite{MarPr99},
\cite[Theorem~1.22]{ChSh09}}]
\label{theorem:Vanya-Kostya-invariants} Let $G\subset\GL_{n+1}(\C)$
be a finite subgroup without reflections. If $n\leqslant 4$, then
the~following are equivalent:
\begin{itemize}
\item the~singularity $\mathbb{C}^{n+1}\slash G$ is exceptional,%
\item the~inequality $\mathrm{lct}(\mathbb{P}^{n},\bar{G})\geqslant (n+2)/(n+1)$ holds,%
\item the~group $G$ is primitive and has no semi-invariants of degree at most $n+1$.%
\end{itemize}
\end{theorem}

The~assertion of Theorem~\ref{theorem:Vanya-Kostya-invariants} is
no longer true if $n\geqslant 5$ (see
\cite[Example~3.25]{ChSh09}).

\begin{remark}
\label{remark:Shokurov-n-2} Let $G\subset\GL_2(\C)$ be a finite
subgroup. Then
$$
\mathrm{lct}\Big(\mathbb{P}^{1},\bar{G}\Big)=\left\{\aligned
&6\ \text{if}\ \bar{G}\cong\A_{5},\\
&4\ \text {if}\ \bar{G}\cong\SS_{4},\\
&2\ \text {if}\ \bar{G}\cong\A_{4},\\
&1\ \text {if}\ \bar{G}\cong\D_{m},\\
&1/2\ \text {if}\ \bar{G}\cong\Z_{m}.\\
\endaligned
\right.
$$
\end{remark}

The main purpose of this paper is to prove
the~following result.

\begin{theorem}
\label{theorem:main} Let $G\subset\SL_6(\mathbb{C})$ be a finite
subgroup. Then the~following are equivalent:
\begin{itemize}
\item the~singularity $\mathbb{C}^{6}\slash G$ is exceptional,%

\item the~inequality
$\mathrm{lct}(\mathbb{P}^{5},\bar{G})\geqslant 7/6$ holds,

\item either $\bar{G}$ is the~Hall--Janko group $\HaJ$ (see \cite{Li68}, \cite{Li70}),
or $G\cong 6.\A_{7}$ and $\bar{G}\cong\A_{7}$.%
\end{itemize}
\end{theorem}

\begin{proof}
The required assertion follows from
Theorems~\ref{theorem:dim-6-invariants}, \ref{theorem:Hall-Janko},
\ref{theorem:A7} and Lemma~\ref{lemma:Segre}.
\end{proof}

As far as we know, Theorem~\ref{theorem:main} gives the first
appearance of the~Hall--Janko group $\HaJ$ in algebraic geometry.
The assertion of Theorem~\ref{theorem:main} gives new examples of
normalized K\"ahler--Ricci iterations that converge to
the~Fubini--Study metric on $\mathbb{P}^5$ (see~\cite{Rub08},
cf.~\cite[Question~1.9]{ChSh09}). Furthermore, it follows from
Corollary~\ref{corollary:criterion} that
Theorem~\ref{theorem:main} can be considered as a~classification
of six-dimensional exceptional quotient singularities.

\begin{remark}
Suppose that $\bar{G}\subset\PGL_6(\mathbb{C})$. If
$\bar{G}\cong\HaJ$, then there exist two subgroups
in~$\SL_6(\mathbb{C})$ whose images in $\PGL_6(\mathbb{C})$
coincide with $\bar{G}$. Namely, one of them is isomorphic
to~$2.\HaJ$, and another one is isomorphic to the~extension of
the~subgroup $2.\HaJ\subset\SL_6(\mathbb{C})$ by a~scalar matrix
with non-zero entries equal to a~primitive root of unity of degree
$6$. Moreover, up to conjugation $\PGL_6(\mathbb{C})$ contains a
unique subgroup isomorphic to $\HaJ$. On the other hand,~$\PGL_6(\mathbb{C})$
contains non-conjugate subgroups isomorphic
to $\mathbb{A}_{7}$. Furthermore, if $\bar{G}\cong\mathbb{A}_{7}$
and the singularity $\mathbb{C}^{6}\slash G$ is exceptional, then we must
necessarily  have $G\cong 6.\A_{7}$, which uniquely determines the
subgroup $\bar{G}\subset\PGL_6(\mathbb{C})$ up to conjugation.
Other alternatives for a minimal lift $G$ of a primitive group
$\A_7\cong\bar{G}\subset\PGL_6(\C)$ are
$G\cong 3.\A_7$ and $G\cong\A_7$ (see Theorem~\ref{theorem:Feit}),
which happens for two other classes of subgroups
$\A_7\cong\bar{G}\subset\PGL_6(\C)$. In the latter cases the singularity
$\C^6\slash G$ is not exceptional.
\end{remark}

Finally, we prove the following surprising result.

\begin{theorem}[{cf.~\cite[Example~3.13]{ChSh09}}]
\label{theorem:dim-7}
There are no exceptional quotient singularities of dimension $7$.
\end{theorem}

\medskip

The plan of the paper is as follows. In
Section~\ref{section:preliminaries} we collect well known
auxiliary results. In Section~\ref{section:6-dim} we
show that apart from the singularities
related to the groups $6.\A_7$ and $2.\HaJ$
all other six-dimensional quotient singularities are
not exceptional. In
Section~\ref{section:6} we prove the exceptionality of the
singularities related to the groups $6.\A_7$ and $2.\HaJ$ thus
completing the proof of Theorem~\ref{theorem:main}. Finally, in
Section~\ref{section:conclusion} we prove
Theorem~\ref{theorem:dim-7}. In Appendix~\ref{section:alpha} we
introduce a new invariant of a Kawamata log terminal singularity
based on the classical $\alpha$-invariant of Tian.

Throughout the paper we use usual notation for cyclic, dihedral, symmetric
and alternating groups, as well as for standard algebraic groups.
For a group $\Gamma$ we denote by $k.\Gamma$ a (non-trivial)
central extension of $\Gamma$ by the central subgroup $\Z_k$ (this
might be non-unique).

Many of the computations with the characters of large
finite groups we need
are too complicated to make by hand (namely, those mentioned in
the proofs of
Theorem~\ref{theorem:dim-6-invariants},~\ref{theorem:Hall-Janko}
and~\ref{theorem:A7} and in
Remark~\ref{remark:HJ-semiinvariants}).
In such cases we used the Magma software~\cite{Magma}.

\medskip

The authors would like to thank G.\,Robinson for numerous useful
explanations and comments, and T.\,Dokchitser and A.\,Khoroshkin for
computational support.

\section{Preliminaries}%
\label{section:preliminaries}

Let $X$ be a~variety with at most Kawamata log terminal
singularities (see \cite[Definition~3.5]{Ko97}), let $B_{X}$ be an
effective $\mathbb{Q}$-divisor on the~variety $X$~such that
$(X,B_{X})$ is log canonical. Then
$$
B_{X}=\sum_{i=1}^{r}a_{i}B_{i},
$$
where $a_{i}\in\mathbb{Q}_{\geqslant 0}$, and $B_{i}$ is a~prime
Weil divisor on the~variety $X$.

Let $\pi\colon\bar{X}\to X$ be a~birational morphism such that
$\bar{X}$ is smooth. Then
$$
K_{\bar{X}}+\sum_{i=1}^{r}a_{i}\bar{B}_{i}\sim_{\mathbb{Q}}\pi^{*}\Big(K_{X}+B_{X}\Big)+\sum_{i=1}^{m}d_{i}E_{i},
$$
where $\bar{B}_{i}$ is the~proper transforms of the~divisor
$B_{i}$ on the~variety $\bar{X}$, and $E_{i}$ is an exceptional
divisor of the~morphism $\pi$, and $d_{i}$ is a~rational number.
We may assume that
$$
\left(\bigcup_{i=1}^{r}\bar{B}_{i}\right)\bigcup\left(\bigcup_{i=1}^{m}E_{i}\right)
$$
is a~divisor with simple normal crossing. Put
$$
\mathcal{I}\Big(X, B_{X}\Big)=\pi_{*}\mathcal{O}_{\bar{X}}\Bigg(\sum_{i=1}^{m}\lceil d_{i}\rceil E_{i}-\sum_{i=1}^{r}\lfloor a_{i}\rfloor B_{i}\Bigg).%
$$

\begin{theorem}[{\cite[Theorem~9.4.8]{La04}}]
\label{theorem:Shokurov-vanishing} Let $H$ be a~nef and big
$\mathbb{Q}$-divisor on $X$ such that
$$
K_{X}+B_{X}+H\equiv D,
$$
where $D$ is a~Cartier divisor on the~variety $X$. Then
$H^{i}(\mathcal{I}(X, B_{X})\otimes D)=0$ for every $i\geqslant
1$.
\end{theorem}

Let $\mathcal{L}(X, B_{X})$ be a~subscheme that corresponds to
the~ideal sheaf $\mathcal{I}(X, B_{X})$. Put
$$
\mathrm{LCS}\Big(X, B_{X}\Big)=\mathrm{Supp}\Bigg(\mathcal{L}\Big(X, B_{X}\Big)\Bigg).%
$$

The~subscheme $\mathcal{L}(X,B_{X})$ is reduced, because
$(X,B_{X})$ is log canonical. Note that
\begin{itemize}
\item $\mathcal{I}(X, B_{X})$ is known as the~multiplier ideal sheaf (see \cite[Section~9.2]{La04}),%
\item $\mathcal{L}(X, B_{X})$ is known as the~log canonical singularities subscheme (see \cite[Definition~2.5]{ChSh08c}),%
\item $\mathrm{LCS}(X,B_{X})$ is known as the~locus of log canonical singularities (see \cite[Definition~3.14]{Sho93}).%
\end{itemize}

Let $Z$ be a~center of log canonical singularities of the~log pair
$(X, B_{X})$ (see \cite[Definition~1.3]{Kaw97}), and let
$\mathbb{LCS}(X, B_{X})$ be the~set of all centers of log
canonical singularities of the~log pair $(X, B_{X})$.

\begin{lemma}[{\cite[Proposition~1.5]{Kaw97}}]
\label{lemma:centers} Let $Z^{\prime}$ be an~element of the~set
$\mathbb{LCS}(X, B_{X})$ such that
$$
\varnothing\ne Z\cap Z^{\prime}=\sum_{i=1}^{k}Z_{i},
$$
where $Z_{i}\subsetneq Z$ is an irreducible subvariety. Then
$Z_{i}\in\mathbb{LCS}(X, B_{X})$ for every $i\in\{1,\ldots,k\}$.
\end{lemma}

Suppose that $Z$ is a~minimal center in $\mathbb{LCS}(X, B_{X})$
(see \cite{Kaw97}, \cite{Kaw98}).

\begin{theorem}[{\cite[Theorem~1]{Kaw98}}]
\label{theorem:Kawamata}
The~variety $Z$ is normal and has at most rational singularities.
For every ample
$\mathbb{Q}$-Cartier $\mathbb{Q}$-divisor $\Delta$ on $X$
there exists an~effective $\mathbb{Q}$-divisor $B_{Z}$ on
the~variety~$Z$ such that
$$
\Big(K_{X}+B_{X}+\Delta\Big)\Big\vert_{Z}\sim_{\mathbb{Q}} K_{Z}+B_{Z},%
$$
and $(Z,B_{Z})$ has Kawamata log terminal singularities.
\end{theorem}

\begin{remark}
\label{remark:RR} In the notation and assumptions of
Theorem~\ref{theorem:Kawamata}, suppose that $K_{X}+B_{X}+\Delta\qlin D$,
where $D$ is a~Cartier divisor on $X$. Put $H=D\vert_{Z}$. Let
$\nu\colon\bar{Z}\to Z$ be a~desingularization.~Then
$$
h^{0}\Big(\mathcal{O}_{Z}\big(H\big)\Big)=\chi\Big(\mathcal{O}_{Z}\Big(H\big)\Big)=\chi\Bigg(\mathcal{O}_{\bar{Z}}\Big(\nu^{*}\big(H\big)\Big)\Bigg)%
$$
by Theorem~\ref{theorem:Shokurov-vanishing}, because $Z$ has at
most rational singularities by Theorem~\ref{theorem:Kawamata}.
\end{remark}

Let $\bar{G}\subseteq\mathrm{Aut}(X)$ be a~finite subgroup.
Suppose that  $B_{X}$ is $\bar{G}$-invariant. Then
$g(Z)\in\mathbb{LCS}(X,B_{X})$
for every $g\in\bar{G}$, and the~locus $\mathrm{LCS}(X,B_{X})$ is
$\bar{G}$-invariant. It follows from Lemma~\ref{lemma:centers}
that
$$
g\big(Z\big)\cap g^{\prime}\big(Z\big)\ne \varnothing\iff g\big(Z\big)=g^{\prime}\big(Z)%
$$
for every $g\in\bar{G}\ni g^{\prime}$, because $Z$ is a~minimal
center in $\mathbb{LCS}(X, B_{X})$.

\begin{lemma}
\label{lemma:Kawamata-Shokurov-trick} Suppose that the~divisor
$B_{X}$ is ample. Let $\epsilon$ be an~arbitrary rational number
such that $\epsilon>1$. Then there exists
an~effective~$\bar{G}$-in\-va\-riant $\mathbb{Q}$-divisor
$D$~on~the~variety~$X$~such~that
$$
\mathbb{LCS}\Big(X, D\Big)=\bigcup_{g\in\bar{G}}\Big\{g\big(Z\big)\Big\},%
$$
the~log pair $(X,D)$ is log canonical, and the~equivalence
$D\sim_{\mathbb{Q}} \epsilon(B_{X})$ holds.
\end{lemma}

\begin{proof}
See the~proofs of \cite[Theorem~1.10]{Kaw97},
\cite[Theorem~1]{Kaw98}, \cite[Lemma~2.8]{ChSh09}.
\end{proof}

Suppose that $X\cong\mathbb{P}^{n}$. Let $H$ be a~hyperplane in
$\mathbb{P}^{n}$. Suppose that
$$
\mathbb{LCS}\Big(X, B_{X}\Big)=\bigcup_{g\in\bar{G}}\Big\{g\big(Z\big)\Big\},%
$$
and let $Y$ be the~$\bar{G}$-orbit of the~subvariety
$Z\subset\mathbb{P}^{n}$.

\begin{lemma}
\label{lemma:degree} Put $s=n-\mathrm{dim}(Y)$ and
$$
r=\left\{\aligned
&\lceil\mu-s-1\rceil+1\ \text{if}\ \mu\in\mathbb{Z},\\
&\lceil\mu-s-1\rceil\ \text {if}\ \mu\not\in\mathbb{Z},\\
\endaligned
\right.
$$
where $\mu\in\mathbb{Q}$ such that $B_{X}\qlin \mu H$. Then
$r\geqslant 0$ and
$$
\mathrm{deg}\big(Y\big)\leqslant {s+r\choose r}.%
$$
\end{lemma}

\begin{proof}
Let $\Pi\subset\mathbb{P}^{n}$ be a~general linear subspace of
dimension $s$. Put
$$
D=B_{X}\Big\vert_{\Pi}
$$
and $\Lambda=H\cap \Pi$. Then $\mathrm{deg}(Y)=|Y\cap \Pi|$ and
$\mathrm{LCS}(\Pi,D)=Y\cap\Pi$. One has
$$
K_{\Pi}+D\qlin \Big(\mu-s-1\Big)\Lambda.
$$

It follows from Theorem~\ref{theorem:Shokurov-vanishing} that
there is an exact sequence of cohomology groups
$$
0\longrightarrow H^{0}\Bigg(\mathcal{O}_{\Pi}\big(r\Lambda\big)\otimes\mathcal{I}\Big(\Pi,D\Big)\Bigg)\longrightarrow H^{0}\Big(\mathcal{O}_{\Pi}\big(r\Lambda\big)\Big)\longrightarrow H^{0}\Big(\mathcal{O}_{\mathcal{L}(\Pi,D)}\Big)\longrightarrow 0,%
$$
and $\mathrm{Supp}(\mathcal{L}(\Pi,D))=\mathrm{LCS}(\Pi,D)=Y\cap
\Pi\ne\varnothing$. Therefore, we see that $r\geqslant 0$ and
$$
\mathrm{deg}\big(Y\big)=\big|Y\cap \Pi\big|\leqslant h^{0}\Big(\mathcal{O}_{\mathcal{L}(\Pi,D)}\Big)\leqslant h^{0}\Big(\mathcal{O}_{\Pi}\big(r\Lambda\big)\Big)=h^{0}\Big(\mathcal{O}_{\mathbb{P}^{s}}\big(r\big)\Big)={s+r\choose r},%
$$
which completes the~proof.
\end{proof}

Let $G$ be a~finite subgroup in $\GL_{n+1}(\mathbb{C})$ such that
$\bar{G}=\phi(G)$, where
$\phi\colon\GL_{n+1}(\mathbb{C})\to\mathrm{Aut}(\mathbb{P}^{n})\cong
\PGL_{n+1}(\mathbb{C})$ is~the~natural
projection.

\begin{lemma}
\label{lemma:invariant-quadric} If $G$ is conjugate to a~subgroup
in $\GL_{n+1}(\mathbb{R})$, then $G$ has
an~invariant~of~degree~$2$.
\end{lemma}

\begin{proof}
If $G$ is conjugate to a~subgroup in $\GL_{n+1}(\mathbb{R})$, then
there exists a (real positive definite) $G$-invariant inner
product, which gives a non-trivial $G$-invariant element
in $\mathrm{Sym}^2(\C^{n+1})$.
\end{proof}

\begin{lemma}
\label{lemma:normal-extension} Suppose that there exists a~normal
subgroup $F\subset G$ such that $G/F$~is~abelian, and $F$ has
an~invariant of degree $d$. Then $G$ has a~semi-invariant of
degree $d$.
\end{lemma}

\begin{proof}
Let $V$ be a~space of invariants of the~group $F$ of degree $d$.
Then the~group $G/F$ naturally acts on the~space $V$. Since
the~group $G/F$ is abelian, it has a~one-dimensional invariant
subspace, which gives a~required semi-invariant of the~subgroup
$G$.
\end{proof}

Let $G_{1}\subset\SL_2(\mathbb{C})$
and~\mbox{$G_{2}\subset\SL_l(\mathbb{C})$} be finite subgroups, let
$\mathbb{M}$ be the~vector space of~$2\times l$-matrices with entries
in~$\mathbb{C}$. For every $(g_{1},g_{2})\in G_{1}\times G_{2}$ and
every $M\in\mathbb{M}$, put
$$
\Big(g_{1},g_{2}\Big)\big(M\big)=g_{1}Mg_{2}^{-1}\in\mathbb{M}\cong\mathbb{C}^{2l},%
$$
which induces a~homomorphism $\varphi\colon G_{1}\times
G_{2}\to\SL_{2l}(\mathbb{C})$. Note that
$|\mathrm{ker}(\varphi)|\leqslant 2$ if $n$ is even, and~$\varphi$
is a~monomorphism if $n$ is odd. Suppose  that $n=2l-1\geqslant
3$.

\begin{lemma}[{\cite[Lemma~3.24]{ChSh09}}]
\label{lemma:Segre} Suppose that $G=\varphi(G_{1}\times G_{2})$.
Then $\mathrm{lct}(\mathbb{P}^{n}, \bar{G})<1$.
\end{lemma}

\begin{proof}
Put $s=l-1$. Let
$\psi\colon\mathbb{P}^{1}\times\mathbb{P}^{s}\to\mathbb{P}^{n}$ be
the~Segre embedding. Put
$Y=\psi(\mathbb{P}^{1}\times\mathbb{P}^{s})$ and let $\mathcal{Q}$
be the~linear system consisting of all quadric hypersurfaces in
$\mathbb{P}^{n}$ that pass through the~subvariety $Y$. Then
$\mathcal{Q}$ is a~non-empty $\bar{G}$-invariant linear system.
The~log pair $(\mathbb{P}^{n}, l\mathcal{Q})$ is not log-canonical
along $Y$. Now it  follows from \cite[Theorem~4.8]{Ko97} that
$\mathrm{lct}(\mathbb{P}^{n}, \bar{G})<1$.
\end{proof}

\section{Six-dimensional case}
\label{section:6-dim}

Let $G$ be a~finite subgroup in $\SL_{n+1}(\mathbb{C})$. Put
$V=\mathbb{C}^{n+1}$.

\begin{definition}[{see \cite[\S 1]{MaWo93}}]
\label{definition:quasi-primitive} The~subgroup $G$ is
quasiprimitive if the~following conditions hold:
\begin{itemize}
\item the~vector space $V$ is an irreducible representation of the~group $G$,%
\item for any nontrivial normal subgroup $N\subseteq G$ one has
$V\cong W^{{}\oplus r}$ as a~representation of $N$, where $W$ is
an irreducible representation of $N$, and $r\ge 1$.
\end{itemize}
\end{definition}

Suppose that $n=5$. Let
$\phi\colon\SL_6(\mathbb{C})\to\mathrm{Aut}(\mathbb{P}^{5})$ be
the~natural projection. Put $\bar{G}=\phi(G)$. We say that
the~subgroup $G$ is the~lift of the~subgroup
$\bar{G}\subset\mathrm{Aut}(\mathbb{P}^{5})\cong\PGL_6(\mathbb{C})$
to~$\SL_6(\mathbb{C})$.

\begin{theorem}[{\cite[\S 3]{Li71}}]
\label{theorem:Feit} Suppose that $G$ is
quasiprimitive.~Then~there~exists a~lift of
the~sub\-group~$\bar{G}\subset\mathrm{Aut}(\mathbb{P}^{5})$ to
$\SL_6(\C)$ that is contained in the~following list:
\begin{itemize}
\item[(I)]
\begin{itemize}
\item[(i)] a~subgroup of the group $\SL_6(\mathbb{C})$ that satisfies the~hypotheses of Lemma~\ref{lemma:Segre},%
\item[(ii)] a~certain subgroup of a~subgroup described in I(i) (see~\cite[\S 3]{Li71} for details),%
\end{itemize}
\item[(II)] $\SL_2(\F_5)$, %
\item[(III)] $2.\SS_5$, %
\item[(IV)]
\begin{itemize}
\item[(i)] $3.\A_6$, %
\item[(ii)] an~extension of the~subgroup described in IV(i) by an automorphism of order $2$,%
\end{itemize}
\item[(V)] $6.\A_6$,%
\item[(VI)] $\A_7$ or $\SS_7$,%
\item[(VII)] $3.\A_7$,%
\item[(VIII)] $6.\A_7$,%
\item[(IX)]%
\begin{itemize}
\item[(i)] $\PSL_2(\F_7)$, %
\item[(ii)] $\PGL_2(\F_7)$,%
\end{itemize}
\item[(X)]
\begin{itemize}
\item[(i)] $\SL_2(\F_7)$, %
\item[(ii)] an~extension of the~subgroup described in X(i) by an automorphism of order $2$,%
\end{itemize}
\item[(XI)] $\SL_2(\F_{11})$,%
\item[(XII)] $\SL_2(\F_{13})$,%
\item[(XIII)]
\begin{itemize}
\item[(i)] $\PSp_4(\F_3)$,%
\item[(ii)] an~extension of the~subgroup described in XIII(i) by an automorphism of order $2$,%
\end{itemize}
\item[(XIV)]
\begin{itemize}
\item[(i)] $\SU_3(\F_3)$, %
\item[(ii)] an~extension of the~subgroup described in XIV(i) by an automorphism of order $2$,%
\end{itemize}
\item[(XV)]
\begin{itemize}
\item[(i)] $6.\PSU_4(\F_3)$, %
\item[(ii)] an~extension of the~subgroup described in XV(i) by an automorphism of order $2$,%
\end{itemize}
\item[(XVI)] $2.\HaJ$, where $\HaJ$ is the~Hall--Janko group (see \cite{Li68}, \cite{Li70}),%
\item[(XVII)]
\begin{itemize}
\item[(i)] $6.\PSL_3(\F_4)$, %
\item[(ii)] an~extension of the~subgroup described in XVII(i) by an automorphism of order $2$.%
\end{itemize}
\end{itemize}
\end{theorem}

Recall that all primitive subgroups are quasiprimitive (see
\cite[\S 1]{MaWo93}).

\medskip
The main purpose of this section is to prove the following result.

\begin{theorem}
\label{theorem:dim-6-invariants} If  $G$ is primitive, then $G$
has a semi-invariant of degree~at~most~$6$~unless there exists a
lift of $\bar{G}$ to $\SL_6(\C)$ that is a~group of type I, VIII
or XVI in the~notation of Theorem~\ref{theorem:Feit}.
\end{theorem}

\begin{proof}
Recall that changing a lift of $\bar{G}$ to $\SL_6(\C)$ does not
change the~degrees of semi-invariants, which implies that we may
assume that $G$ is one of the~groups listed in
Theorem~\ref{theorem:Feit}.

If the~subgroup $G$ is of type VI, IX(i) or XIII(i), then
the~subgroup $G$ is conjugate to a~subgroup of $\SL_6(\mathbb{Q})$
(see~\cite{Atlas}), and hence $G$ has an~invariant of
degree $2$ by Lemma~\ref{lemma:invariant-quadric}.

If the~subgroup $G$ is of type XV(i), then $G$ is a~subgroup of
the~Mitchell group $6.\PSU_4(\F_3).2$, which is the~group $\# 34$
in~\cite[Table~VII]{SheTo54}, and hence the~subgroup $G$
has~an~invariant~of~degree~$6$, because the~Mitchell group
has~an~invariant~of~degree~$6$ (see~\cite[Table~VII]{SheTo54}).

If the~subgroup $G$ is of type II, V, VII, X(i), XI, XII, XIV(i)
or XVII(i), then the~minimal degree $d_{min}$ of the~invariants of
the~subgroup $G$ is given in the~following table:
\begin{center}
{\renewcommand\arraystretch{1.3}
\begin{tabular}{|c|c|c|c|c|c|c|c|c|} \hline $G$ & $2.\A_5$
& $6.\A_6$ & $3.\A_7$ & $\SL_2(\F_7)$
& $\SL_2(\F_{11})$ & $\SL_2(\F_{13})$ & $\SU_3(\F_3)$ & $6.\PSL_3(\F_4)$\\
\hline
Type & II & V & VII & X(i) & XI & XII & XIV(i) & XVII(i)\\
\hline
$d_{min}$ & $4$ & $6$ & $3$ & $4$ & $4$ & $4$ & $6$ & $6$  \\
\hline
\end{tabular}}
\end{center}

If the~subgroup $G$ is a~subgroup of type IV(i), then $G$ is
a~subgroup of a~quasiprimitive subgroup~of~type~VII, which implies
that the~subgroup $G$ has an invariant of degree $3$.

If the subgroup $G$ is a~subgroup of type III, then it has
a~normal subgroup isomorphic~to~$2.\A_5$, which implies that $G$
has a~semi-invariant of degree~$4$ by
Lemma~\ref{lemma:normal-extension}.

Arguing as in the case of a~subgroup of type III, we see that
the~subgroup $G$ has a~semi-in\-va\-ri\-ant of degree $3$, $2$,
$4$, $2$, $6$, $6$ or $6$ in the case when the~subgroup
$G\subset\SL_6(\mathbb{C})$ is a~quasiprimitive subgroup of type
IV(ii), IX(ii), X(ii), XIII(ii), XIV(ii), XV(ii) or XVII(ii),
respectively.
\end{proof}

\begin{remark}\label{remark:HJ-semiinvariants} In the~notation of
Theorem~\ref{theorem:Feit},
if $G$ is a primitive subgroup of type VIII or XVI, then a direct
computation shows that the~minimal degree of the~semi-invariants
of $G$ equals $12$.
\end{remark}

\section{Exceptional cases}
\label{section:6}

Let $G$ be a~subgroup in $\SL_6(\mathbb{C})$. Define $V$ and
$\bar{G}$ as in Section~\ref{section:6-dim}.

\begin{remark}
\label{remark:semiinvariants-vs-invariants} If the~group $\bar{G}$
is a~simple non-abelian group such that
$Z(G)\subseteq [G,G]$,
where~$Z(G)$ and~\mbox{$[G,G]$} denote the center and the commutator of
the subgroup~$G$, respectively, then every semi-invariant of
the~group~$G$ is its invariant.
\end{remark}

\begin{theorem}
\label{theorem:Hall-Janko} Suppose that $\bar{G}\cong\HaJ$ is
the~Hall--Janko group. Then
$\mathrm{lct}(\mathbb{P}^{5},\bar{G})\geqslant 7/6$.
\end{theorem}
\begin{proof}
We may assume that $G\cong 2.\HaJ$ (see
Theorem~\ref{theorem:Feit}). Then $Z(G)\subseteq [G,G]$.

Suppose that~$\mathrm{lct}(\mathbb{P}^{5},\bar{G})<7/6$. Then there
is an~effective $\bar{G}$-invariant $\mathbb{Q}$-divisor
$$
D\sim_{\mathbb{Q}}-K_{\mathbb{P}^{5}}\sim\mathcal{O}_{\mathbb{P}^{5}}\big(6\big),
$$
and there is a~positive rational number $\lambda<7/6$ such that
$(\mathbb{P}^{5},\lambda D)$ is strictly log canonical.

Let $S$ be a~minimal center in
$\mathbb{LCS}(\mathbb{P}^{5},\lambda D)$, let $Z$ be
the~$\bar{G}$-orbit of the~subvariety $S\subset\mathbb{P}^{5}$,
and let $r$ be the~number of irreducible components of
the~subvariety $Z$. We may assume that
$$
\mathbb{LCS}\Big(\mathbb{P}^{5},\lambda D\Big)=\bigcup_{g\in\bar{G}}\Big\{g\big(S\big)\Big\}%
$$
by Lemma~\ref{lemma:Kawamata-Shokurov-trick}. Then
$\mathrm{Supp}(Z)=\mathrm{LCS}(\mathbb{P}^{5},\lambda D)$. It
follows from Lemma~\ref{lemma:centers} that
$$
g\big(S\big)\cap g^{\prime}\big(S\big)\ne \varnothing\iff g\big(S\big)=g^{\prime}\big(S)%
$$
for every $g\in\bar{G}\ni g^{\prime}$. Then
$\mathrm{deg}(Z)=r\mathrm{deg}(S)$.

It follows from Remark~\ref{remark:HJ-semiinvariants} that the
subgroup $G$ does not have invariants of degree up to $6$, which
immediately implies that $\mathrm{dim}(S)\ne 4$ by
Remark~\ref{remark:semiinvariants-vs-invariants}.

Let $\mathcal{I}$ be the~multiplier ideal sheaf of the~log pair
$(\mathbb{P}^{5},\lambda D)$, and let $\mathcal{L}$ be the~log
canonical singularities subscheme of the~log pair
$(\mathbb{P}^{5},\lambda D)$. By
Theorem~\ref{theorem:Shokurov-vanishing}, there is an~exact
sequence
$$
0\longrightarrow H^{0}\Big(\mathcal{O}_{\mathbb{P}^{5}}\big(n\big)\otimes\mathcal{I}\Big)\longrightarrow H^{0}\Big(\mathcal{O}_{\mathbb{P}^{5}}\big(n\big)\Big)\longrightarrow H^{0}\Big(\mathcal{O}_{\mathcal{L}}\otimes\mathcal{O}_{\mathbb{P}^{5}}\big(n\big)\Big)\longrightarrow 0%
$$
for every $n\geqslant 1$. A direct computation shows that
$\mathrm{Sym}^{n}(V)$ is an~irreducible representation of
the~group $G$ for all $n\leqslant 5$. Hence, we see that
$$
h^{0}\Big(\mathcal{O}_{\mathbb{P}^{5}}\big(n\big)\otimes\mathcal{I}\Big)=0
$$
for every $n\in\{1,2,3,4,5\}$. Note that $Z=\mathcal{L}$, because
$(\mathbb{P}^{5},\lambda D)$ is log canonical. Thus, we have
\begin{equation}
\label{equation:exact-sequence-n-6}
h^{0}\Big(\mathcal{O}_{Z}\otimes\mathcal{O}_{\mathbb{P}^{5}}\big(n\big)\Big)=h^{0}\Big(\mathcal{O}_{\mathbb{P}^{5}}\big(n\big)\Big)={5+n\choose n}%
\end{equation}
for every $n\in\{1,2,3,4,5\}$. In particular, we see that $r\le 6$,
because
$$
rh^{0}\Big(\mathcal{O}_{S}\otimes\mathcal{O}_{\mathbb{P}^{5}}\big(1\big)\Big)=h^{0}\Big(\mathcal{O}_{Z}\otimes\mathcal{O}_{\mathbb{P}^{5}}\big(1\big)\Big)=6,%
$$
which implies that $r=1$, because $\bar{G}$ has no nontrivial maps
to $\mathrm{S}_r$ for $2\le r\le 6$.

Note that it follows from the~equality $r=1$ that
$\mathrm{dim}(S)\ne 0$.

Let $H$ be a~hyperplane section of the~variety
$S\subset\mathbb{P}^{5}$. It follows from
Theorem~\ref{theorem:Kawamata}~that
the~variety $S$ is normal and has at most rational singularities,
and there are an~effective \mbox{$\mathbb{Q}$-divisor}~$B_{S}$ and
an~ample \mbox{$\mathbb{Q}$-divisor}~$\Delta$ on the~surface~$S$ such
that
$K_{S}+B_{S}+\Delta\qlin H$,
and the~log pair $(S,B_{S})$ has Kawamata log terminal
singularities.

Using the Riemann--Roch theorem and Remark~\ref{remark:RR}, we see
that $\chi(\mathcal{O}_{S}(nH))$ is a polynomial in $n$ of degree
at most $\dim(S)$  such that
$$
\chi\Big(\mathcal{O}_{S}\big(nH\big)\Big)=h^{0}\Big(\mathcal{O}_{S}\big(nH\big)\Big)
$$
for any $n\geqslant 1$. On the other hand, it follows from
$(\ref{equation:exact-sequence-n-6})$ that
$$
\chi\Big(\mathcal{O}_{S}\big(nH\big)\Big)=\left\{\aligned
&6\ \text{if}\ n=1,\\
&21\ \text{if}\ n=2,\\
&56\ \text{if}\ n=3,\\
&126\ \text{if}\ n=4,\\
&252\ \text{if}\ n=5,\\
\endaligned
\right.
$$
which gives an~inconsistent system of linear equations on the
coefficients of the polynomial~$\chi(\mathcal{O}_{S}(nH))$, since
$\dim(S)\leqslant 3$.
\end{proof}

\begin{theorem}
\label{theorem:A7} Suppose that $G\cong 6.\A_{7}$. Then
$\mathrm{lct}(\mathbb{P}^{5},\bar{G})\geqslant 7/6$.
\end{theorem}

\begin{proof}
Suppose that~$\mathrm{lct}(\mathbb{P}^{5},\bar{G})<7/6$. Then
there is an~effective $\bar{G}$-invariant $\mathbb{Q}$-divisor
$$
D\sim_{\mathbb{Q}}-K_{\mathbb{P}^{5}}\sim\mathcal{O}_{\mathbb{P}^{5}}\big(6\big),
$$
and there is a~positive rational number $\lambda<7/6$ such that
$(\mathbb{P}^{5},\lambda D)$ is strictly log canonical.

Arguing as in the~proof of Theorem~\ref{theorem:Hall-Janko}, we
may assume that
$$
\mathbb{LCS}\Big(\mathbb{P}^{5},\lambda D\Big)=\bigcup_{g\in\bar{G}}\Big\{g\big(S\big)\Big\},%
$$
where $S$ is a~minimal center of log canonical singularities of
the~log pair $(\mathbb{P}^{5},\lambda D)$.

Let $Z$ be the~$\bar{G}$-orbit of the~subvariety
$S\subset\mathbb{P}^{5}$. Then
$\mathrm{LCS}(\mathbb{P}^{5},\lambda D)=\mathrm{Supp}(Z)$.

It follows from Remark~\ref{remark:HJ-semiinvariants} that the
subgroup $G$ does not have invariants of degree up to $6$, which
immediately implies that $\mathrm{dim}(S)\ne 4$ by
Remark~\ref{remark:semiinvariants-vs-invariants}.

Let $\mathcal{I}$ be the~multiplier ideal sheaf of the~log pair
$(\mathbb{P}^{5},\lambda D)$, and let $\mathcal{L}$ be the~log
canonical singularities subscheme of the~log pair
$(\mathbb{P}^{5},\lambda D)$. By
Theorem~\ref{theorem:Shokurov-vanishing}, we have
\begin{equation}
\label{equation:exact-sequence}
\chi\Big(\mathcal{O}_{Z}\big(nH\big)\Big)=h^{0}\Big(\mathcal{O}_{Z}\big(nH\big)\Big)=
{5+n\choose n}-h^{0}\Big(\mathcal{O}_{\mathbb{P}^{5}}\big(n\big)\otimes\mathcal{I}\Big),%
\end{equation}
for every $n\geqslant 1$, because $Z=\mathcal{L}$. Put
$q_{n}=h^{0}(\mathcal{O}_{\mathbb{P}^{5}}(n)\otimes\mathcal{I})$
for every $n\geqslant 1$. Then
$$
q_{1}=q_{2}=0,
$$
because $V$ and $\mathrm{Sym}^{2}(V)$ are irreducible
representations of the~group $G$. Hence
\begin{equation}
\label{equation:equality-1-2}
h^{0}\Big(\mathcal{O}_{S}\otimes\mathcal{O}_{\mathbb{P}^{5}}\big(n\big)\Big)=h^{0}\Big(\mathcal{O}_{\mathbb{P}^{5}}\big(n\big)\Big)={5+n\choose n}%
\end{equation}
for $n\in\{1, 2\}$ by $(\ref{equation:exact-sequence})$. In
particular, we see that $r\le 6$, because
$$
rh^{0}\Big(\mathcal{O}_{S}\otimes\mathcal{O}_{\mathbb{P}^{5}}\big(1\big)\Big)=h^{0}\Big(\mathcal{O}_{Z}\otimes\mathcal{O}_{\mathbb{P}^{5}}\big(1\big)\Big)=6,%
$$
which implies that $r=1$ and $Z=S$, because $\bar{G}$
has no nontrivial maps
to $\mathrm{S}_r$ for $2\le r\le 6$.

Note that it follows from the~equality $r=1$ that
$\mathrm{dim}(S)\ne 0$.

Similarly, we see that $q_{3}\in\{0,20,36\}$, because
$$
\mathrm{Sym}^{3}\big(V\big)=T_{36}\oplus T_{20},
$$
where $T_{i}$ is an~irreducible representation of the~group $G$ of
dimension $i$. Thus, we have
\begin{equation}
\label{equation:equality-3}
h^{0}\Big(\mathcal{O}_{S}\otimes\mathcal{O}_{\mathbb{P}^{5}}\big(3\big)\Big)=h^{0}\Big(\mathcal{O}_{\mathbb{P}^{5}}\big(3\big)\Big)-h^{0}\Big(\mathcal{O}_{\mathbb{P}^{5}}\big(3\big)\otimes\mathcal{I}\Big)=56-h^{0}\Big(\mathcal{O}_{\mathbb{P}^{5}}\big(3\big)\otimes\mathcal{I}\Big)\in\big\{20,36,56\big\}%
\end{equation}
by $(\ref{equation:exact-sequence})$. Moreover, one has
$$
\mathrm{Sym}^{4}\big(V\big)=U_{6}\oplus U_{15}\oplus
\hat{U}_{15}\oplus U_{21}^{{}\oplus 2}
\oplus U_{24}\oplus \hat{U}_{24},%
$$
where $U_{i}$ and $\hat{U}_{i}$ are~irreducible representations of
the~group $G$ of dimension $i$. In particular,
\begin{equation}
\label{eq:values-of-h4}
h^{0}\Big(\mathcal{O}_{S}\otimes\mathcal{O}_{\mathbb{P}^{5}}\big(4\big)\Big)=
126-h^{0}\Big(\mathcal{O}_{\mathbb{P}^{5}}\big(4\big)\otimes\mathcal{I}\Big)
\not\in\big\{106,114,\ldots,119,121,\ldots,125\big\}%
\end{equation}
by $(\ref{equation:exact-sequence})$. Finally, one has
$$
\mathrm{Sym}^{5}\big(V\big)=W_{11}^{{}\oplus 2}\oplus
W_{24}^{{}\oplus 2}\oplus \hat{W}_{24}^{{}\oplus 2}\oplus
W_{36}^{{}\oplus 4},%
$$
where $W_{i}$ and $\hat{W}_{i}$ are~irreducible representations of
the~group $G$ of dimension $i$. By
$(\ref{equation:exact-sequence})$,~we~have
\begin{equation}\label{eq:values-of-h5}
h^{0}\Big(\mathcal{O}_{S}\otimes\mathcal{O}_{\mathbb{P}^{5}}\big(5\big)\Big)=
252-h^{0}\Big(\mathcal{O}_{\mathbb{P}^{5}}\big(5\big)\otimes\mathcal{I}\Big)
\not\in\big\{66, 171, 179\big\}.%
\end{equation}

Let $H$ be a~hyperplane section of the~variety
$S\subset\mathbb{P}^{5}$. It follows from
Theorem~\ref{theorem:Kawamata}~that
the~variety $S$ is normal and has at most rational singularities,
and there are an~effective \mbox{$\mathbb{Q}$-divisor}~$B_{S}$ and
an~ample \mbox{$\mathbb{Q}$-divisor}~$\Delta$ on the~surface~$S$ such
that
$K_{S}+B_{S}+\Delta\qlin H$,
and the~log pair $(S,B_{S})$ has Kawamata log terminal
singularities.

Suppose that $\mathrm{dim}(S)=1$. Then $S$ is a~smooth curve of
genus $g$ such that
$$
\mathrm{deg}\big(H\big)=\mathrm{deg}\big(S\big)>2g-2,
$$
and $\mathrm{deg}(Z)\leqslant 15$ by Lemma~\ref{lemma:degree}. By
the~Riemann--Roch theorem, we~get
$$
h^{0}\Big(\mathcal{O}_{S}\big(nH\big)\Big)=n\mathrm{deg}\big(S\big)-g+1
$$
for every $n\geqslant 1$ (see Remark~\ref{remark:RR}). Using
$(\ref{equation:equality-1-2})$, we see~that
$$
\left\{\aligned
&6=\mathrm{deg}\big(S\big)-g+1,\\
&21=2\mathrm{deg}\big(S\big)-g+1,\\
\endaligned
\right.
$$
which implies that $\mathrm{deg}(S)=15$ and $g=10$. Using
$(\ref{equation:equality-1-2})$ again one obtains
$$5\deg\big(S\big)-g+1=66,$$
which is impossible by~$(\ref{eq:values-of-h5})$.

Suppose that $\mathrm{dim}(S)=2$. Using the~Riemann--Roch theorem
and Remark~\ref{remark:RR}, we have
\begin{equation}
\label{equation:RR-2}
h^{0}\Big(\mathcal{O}_{S}\big(nH\big)\Big)=\chi\Big(\mathcal{O}_{S}\big(nH\big)\Big)=
\frac{n^{2}}{2}\Big(H\cdot H\Big)-\frac{n}{2}\Big(H\cdot K_{S}\Big)+\chi\big(\mathcal{O}_{S}\big)%
\end{equation}
for any $n\geqslant 1$. Thus, using
$(\ref{equation:equality-1-2})$ and $(\ref{equation:equality-3})$,
we see~that
$$
\Big( \mathrm{deg}\big(S\big), H\cdot K_{S}, \chi\big(\mathcal{O}_{S}\big)\Big)\in
\Big\{\big(5,-5, 6\big),\big(15,5, -4\big)\Big\},%
$$
because $H\cdot H=\mathrm{deg}(S)>0$. If $(\mathrm{deg}(S), H\cdot
K_{S}, \chi(\mathcal{O}_{S}))=(15,5, -4)$, then
$$
h^{0}\Big(\mathcal{O}_{S}\big(4H\big)\Big)=8\Big(H\cdot H\Big)-2\Big(H\cdot K_{S}\Big)+
\chi\big(\mathcal{O}_{S}\big)=106,%
$$
which is impossible by $(\ref{eq:values-of-h4})$. If
$(\mathrm{deg}(S), H\cdot K_{S}, \chi(\mathcal{O}_{S}))=(15,5, 6)$,
then
$$
h^{0}\Big(\mathcal{O}_{S}\big(4H\big)\Big)=116,%
$$
which is again impossible by $(\ref{eq:values-of-h4})$.

We see that $\mathrm{dim}(S)=3$. Then $H\cdot H \cdot
H=\mathrm{deg}(S)\geqslant\mathrm{codim}(S)+1=3$, since~$V$
is an irreducible representation of the group~$G$.

Let $H^{\prime}$ be another general hyperplane section of
$S\subset\mathbb{P}^{5}$. Put $C=H\cap H^{\prime}$. Then
$$
-2\leqslant 2g\big(C\big)-2=H\cdot H\cdot K_{S}+2\Big(H\cdot H \cdot H\Big),%
$$
where $g(C)$ is the~genus of the~curve $C$. Thus, we see that
\begin{equation}
\label{equation:A-7-KHH-HHH}
H\cdot H\cdot K_{S}\geqslant-2-2\mathrm{deg}\big(S\big).%
\end{equation}

By the~Riemann--Roch theorem and Remark~\ref{remark:RR}, there is
$\gamma\in\mathbb{Z}$ such that
$$
h^{0}\Big(\mathcal{O}_{S}\big(nH\big)\Big)=\chi\Big(\mathcal{O}_{S}\big(nH\big)\Big)=\frac{n^{3}}{6}\Big(H\cdot H\cdot H\Big)+\frac{n^{2}}{4}\Big(H\cdot H\cdot K_{S}\Big)+\frac{n}{12}\gamma+\chi\big(\mathcal{O}_{S}\big)%
$$
for any $n\geqslant 1$. Put $h_n=h^{0}(\mathcal{O}_{S}(nH))$. Then
\begin{equation}\label{eq:RR-1}
h_4-3h_3+3h_2-h_1=H\cdot H\cdot H,
\end{equation}
and
\begin{equation}\label{eq:RR-2}
h_3-2h_2+h_1=2\Big(H\cdot H\cdot H\Big)+\frac{1}{2}\Big(H\cdot
H\cdot K_S\Big),
\end{equation}
which implies after applying~$(\ref{equation:A-7-KHH-HHH})$
\begin{equation}\label{eq:RR-3}
h_4\leqslant 2h_1-5h_2+4h_3+1.
\end{equation}
Since $H\cdot H\cdot H\geqslant 3$, the~equality~$(\ref{eq:RR-1})$
also implies
\begin{equation}\label{eq:RR-4}
h_4\geqslant 3+h_1-3h_2+3h_3.
\end{equation}

Recall that $h_1=6$, $h_2=20$, $h_3\in\{20, 36, 56\}$.

If $h_3=20$, then~$(\ref{eq:RR-3})$ implies that $h_4\leqslant
-12$, which is a~contradiction.

If $h_3=36$, then~$(\ref{eq:RR-3})$ and $(\ref{eq:RR-4})$ imply
that $52\geqslant h_4\geqslant 54$, which is a~contradiction.

We see that $h_3=56$. Then~$(\ref{eq:RR-4})$ implies that
$h_4\geqslant 114$, so that
$$
h_4\in\big\{120, 126\big\}
$$
by~$(\ref{eq:values-of-h4})$. If $h_4=120$, then~$(\ref{eq:RR-1})$
implies that $H\cdot H\cdot H=9$. Note that
\begin{equation}\label{eq:RR-5}
H\cdot H\cdot H=h_5-3h_4+3h_3-h_2,
\end{equation}
and hence $h_5=171$, which is impossible
by~$(\ref{eq:values-of-h5})$. Thus, we see that $h_4=126$.

It follows from $(\ref{eq:RR-1})$ and $(\ref{eq:RR-5})$ that
$H\cdot H\cdot H=15$ and $h_5=179$, which is impossible
by~$(\ref{eq:values-of-h5})$.
\end{proof}

\section{Seven-dimensional singularities}
\label{section:conclusion}

Let $G$ be a~finite subgroup in $\SL_7(\mathbb{C})$, and let
$\phi\colon\SL_7(\mathbb{C})\to\mathrm{Aut}(\mathbb{P}^{6})$ be
the~natural~projec\-tion.~Put $\bar{G}=\phi(G)$. We say that
the~subgroup $G$ is the~lift of the~subgroup
$\bar{G}$~to~$\SL_7(\mathbb{C})$.

\begin{theorem}[{\cite[Theorem~4.1]{Wa69}, \cite[Theorem~I]{Wa70}}]
\label{theorem:Feit-7} Suppose that the~subgroup $G$ is
quasiprimitive.~Then~there~is a~lift of the~subgroup $\bar{G}$ to
$\SL_7(\C)$ that is contained in the~following~list:
\begin{itemize}
\item[(I)] a~subgroup of the subgroup
$G_7\subset\SL_7(\mathbb{C})$ such that
$$
G_{7}=\mathrm{Norm}_{\SL_7(\C)}(\mathbb{H}_7)
\cong\mathbb{H}_7\rtimes\SL_2\Big(\mathbb{F}_7\Big),
$$
where $\mathbb{H}_7$ is the~Heisenberg group of order $7^3$, and
the~corresponding seven-dimensional representation of the~group
$\mathbb{H}_7$ is any of its $7$-dimensional irreducible representations,
\item[(II)] $\PSL_2(\F_{13})$, %
\item[(III)]
\begin{itemize}
\item[(i)] $\PSL_2(\F_8)$,%
\item[(ii)] an~extension of the~subgroup described in III(i) by an~automorphism of order $3$,%
\end{itemize}
\item[(IV)] $\A_8$ or $\SS_8$,%
\item[(V)]
\begin{itemize}
\item[(i)] $\PSL_2(\F_7)$,%
\item[(ii)] $\PGL_2(\F_7)$,
\end{itemize}
\item[(VI)]
\begin{itemize}
\item[(i)] $\PSU_3(\F_3)$,
\item[(ii)] an~extension of the~subgroup described in VI(i) by an automorphism of order $2$,%
\end{itemize}
\item[(VII)] $\Sp_6(\F_2)$.
\end{itemize}
\end{theorem}

\begin{remark}
\label{remark:conjugate-non-unique} Up to conjugation there are
two primitive subgroups in $\SL_7(\C)$ that are
isomorphic to~\mbox{$\PSL_2(\F_8)$}:
one is conjugate to a~subgroup in $\SL_7(\Q)$, and
another is conjugate to a~subgroup in $\SL_7(\Q(\xi_9))$,
where $\xi_9$ is a~primitive root of unity of degree $9$.
Similarly, up to conjugation there are
two primitive subgroups in $\SL_7(\C)$ that are
isomorphic to~\mbox{$\PSU_3(\F_3)$}:
one is conjugate to a~subgroup in $\SL_7(\Q)$,
another is conjugate to a~subgroup in $\SL_7(\Q(\sqrt{-1}))$.
The detailed information on the corresponding representations
may be found in~\cite{Atlas}.
\end{remark}

The main purpose of this section is to prove the following result.

\begin{theorem}
\label{theorem:dim-7-our} Suppose that $G$ is~quasiprimitive. Then
either $G$ has a~semi-invariant of degree at most~$7$,
or $G$ is a~subgroup of the subgroup
$G_7\subset\SL_7(\mathbb{C})$ (see Theorem~\ref{theorem:Feit-7}).
\end{theorem}

\begin{proof}
Recall that changing a lift of $\bar{G}$ to $\SL_7(\C)$ does not
change the~degrees of semi-invariants, which implies that we may
assume that $G$ is one of the~groups listed in
Theorem~\ref{theorem:Feit-7}.

Note that the~groups of type I have a unique lift to $\SL_7(\C)$.

By Lemma~\ref{lemma:invariant-quadric}, we may assume that $G$ is
not conjugate to a~subgroup~in~$\SL_7(\mathbb{Q})$, which implies
that the~subgroup $G$ is not of type IV,  V(i) or VII
(see~\cite{Atlas}).

If the~subgroup $G$ is of type II, III(i) or VI(i) (cf.
Remark~\ref{remark:conjugate-non-unique}), then the~minimal degree
$d_{min}$ of the~invariants of the~subgroup $G$ is given in
the~following table:

\begin{center}
{\renewcommand\arraystretch{1.3}
\begin{tabular}{|c|c|c|c|c|c|}
\hline
$G$ & $\PSL_2(\F_{13})$ & $\PSL_2(\F_8)$ & $\PSU_3(\F_3)$\\
\hline
Type & II & III(i) & VI(i)\\
\hline $d_{min}$ & $2$ & $2$ & $3$\\
\hline
\end{tabular}}
\end{center}

If the subgroup $G$ is a~subgroup of type III(ii), then it has
a~normal subgroup of index $3$ isomorphic~to~$\PSL_2(\F_8)$,
which implies that $G$ has a~semi-invariant of degree~$2$ by
Lemma~\ref{lemma:normal-extension}.

Arguing as in the case of a~subgroup of type III(ii), we see that
the~subgroup $G$ has a~semi-in\-va\-ri\-ant of degree at most $3$
if $G$ is a~quasiprimitive subgroup of type V(ii) or VI(ii).
\end{proof}

\begin{proof}[{Proof of Theorem~\ref{theorem:dim-7}}]
By~\cite[Lemma~2.2.1(xi)]{Miele} (see also \cite[\S1]{MeRa03}),
the group $G_7\subset\SL_7(\mathbb{C})$ has an invariant of degree~$7$.
Thus, Theorem~\ref{theorem:dim-7-our} implies
Theorem~\ref{theorem:dim-7}.
\end{proof}

\appendix

\section{Alpha-invariant}
\label{section:alpha}

Let $(V\ni O)$ be a~germ of a Kawamata log terminal singularity
(see \cite[Definition~3.5]{Ko97}), and let $\pi\colon W\to V$
be~a~birational morphism such that
\begin{itemize}
\item the~exceptional locus of $\pi$ consists of one irreducible divisor $E\subset W$ such that $O\in\pi(E)$,%

\item the~log pair $(W,E)$ has purely log terminal singularities (see \cite[Definition~3.5]{Ko97}),%

\item the~divisor $-E$ is a~$\pi$-ample $\mathbb{Q}$-Cartier divisor.%
\end{itemize}

\begin{theorem}
\label{theorem:plt-blow-up} The~birational morphism $\pi\colon
W\to V$ does exist.
\end{theorem}

\begin{proof}
The~required assertion follows from
\cite[Proposition~2.9]{Pr98plt}, \cite[Theorem~1.5]{Kud01} and
\cite{BCHM06}.
\end{proof}

The~existence of $\pi$ is obvious if $(V\ni O)$ is a~quotient
singularity (see \cite[Remark~3.15]{ChSh09}) or an isolated
quasi-homogeneous hypersurface singularity.

\begin{definition}[{\cite[Definition~2.1]{Pr98plt}}]
\label{definition:plt-blow-up} We say that $\pi$ is a~plt blow up
of the~germ~$(V\ni O)$.
\end{definition}

\begin{definition}[{\cite[Definition~4.1]{Pr98plt}}]
\label{definition:weakly-exceptional} We say that $(V\ni O)$ is
weakly-exceptional if $\pi$ is unique.
\end{definition}

The~goal of this appendix is to define an~invariant $\alpha(V\ni
O)\in\mathbb{R}$ of the~singularity~$(V\ni O)$, which is a~local
analogue of the~$\alpha$-invariant introduced in \cite{Ti87} and
\cite{TiYa87}.

\begin{lemma}[{see \cite[Theorem~4.9]{Pr98plt}}]
\label{lemma:exceptional} If  $(V\ni O)$ is exceptional,
then $\pi(E)=O$.
\end{lemma}

\begin{lemma}[{\cite[Corollary~1.7]{Kud01},\cite{BCHM06}}]
\label{lemma:weakly-exceptional} If  $(V\ni O)$ is
weakly-exceptional, then $\pi(E)=O$.
\end{lemma}

If $\pi(E)\ne O$, then we put $\alpha(V\ni O)=0$. Suppose, in
addition, that $\pi(E)=O$.

Denote by $R_{1}, \ldots,R_{s}$ the irreducible components of the~locus
$\mathrm{Sing}(W)$ such that~\mbox{$\mathrm{dim}(R_{i})=\mathrm{dim}(V)-2$}
and $R_{i}\subset E$ for every $i\in\{1,\ldots,s\}$. Put
$$
\Delta=\sum_{i=1}^{s}\frac{m_{i}-1}{m_{i}}R_{i},
$$
where $m_{i}$ is the~smallest positive integer such that $m_{i}E$
is~Cartier at a~general point of the~subvariety $R_{i}\subset E$.
(One has $\Delta=\mathrm{Diff}_{E}(0)$ in the~notation of the~paper
\cite{Pr98plt}.)

\begin{lemma}[{\cite[Theorem~7.5]{Ko97}}]
\label{lemma:plt-log-Fano} The variety $E$ is normal, and
the~log pair \mbox{$(E,\mathrm{Diff}_{E}(0))$} is Kawamata log terminal.
\end{lemma}

The~log pair $(E,\Delta)$ is a~log Fano variety, i.e. the~divisor
$-(K_{E}+\Delta)$ is ample. Indeed, the~divisor $-E$ is
$\pi$-ample, and
$$
K_{E}+\Delta\qlin \Big(K_{W}+E\Big)\Big\vert_{E}\qlin
\Big(\pi^{*}\big(K_{V}\big)+
\big(1+a\big)E\Big)\Big\vert_{E}.%
$$
Moreover, one has $a>-1$, because $V$ has Kawamata log terminal
singularities. Put
$$
\mathrm{lct}\Big(E, \Delta\Big)=\mathrm{sup}\left\{\lambda\in\mathbb{Q}\ \left|%
\aligned
&\text{the~log pair}\ \left(E, \Delta+\lambda D\right)\ \text{is log canonical}\\
&\text{for any effective $\mathbb{Q}$-divisor}\ D\qlin -\Big(K_{E}+\Delta\Big)\\
\endaligned\right.\right\}.%
$$

\begin{theorem}[{\cite[Theorem~2.1]{Kud01}}]
\label{theorem:weakly-exceptional-criterion} The singularity
$(V\ni O)$ is weakly-exceptional if and only if
the~inequality $\mathrm{lct}(E, \Delta)\geqslant 1$ holds.
\end{theorem}

Note that the~real number $\mathrm{lct}(E, \Delta)$ is
an~algebraic counter-part of the~so-called $\alpha$-invariant
introduced in \cite{Ti87} and \cite{TiYa87}
(cf.~\cite[Theorem~A.3]{ChSh08c}). Put
$$
\alpha\big(V\ni O\big)=\left\{\aligned
&\mathrm{lct}\big(E, \Delta\big)\ \text{if}\ \mathrm{lct}\big(E, \Delta\big)\geqslant 1,\\
&0\ \text{if}\ \mathrm{lct}\big(E, \Delta\big)<1.\\
\endaligned
\right.
$$

\begin{definition}
\label{definition:alpha-invariant} We say that $\alpha(V\ni O)$ is
the~alpha-invariant of the~singularity $(V\ni O)$.
\end{definition}

Note that $\alpha(V\ni O)\ne 0$ $\iff$ $\alpha(V\ni O)\geqslant 1$
$\iff$ $(V\ni O)$ is weakly-exceptional.

\begin{example}[{\cite[Lemma~5.2]{ChShPa08}}]
\label{example:alpha-invariant} Suppose that $(V\ni O)$ is
an isolated quasi-homogeneous hypersurface singularity
$$
z^2t+yt^2+xy^4+x^8z=0\subset\mathbb{C}^{4}\cong\mathrm{Spec}\Big(\mathbb{C}\big[x,y,z,t\big]\Big),%
$$
where $O\in V$ is given by $x=y=z=t=0$. Then $\alpha(V\ni
O)=33/4$.
\end{example}

\begin{theorem}[{\cite[Theorem~4.9]{Pr98plt}}]
\label{theorem:exceptional-criterion} The
the~singularity $(V\ni O)$ is exceptional if and only if
for every effective $\mathbb{Q}$-divisor $D$ on the~variety
$E$ such that~\mbox{$D\qlin -(K_{E}+\Delta)$}
the~log pair $(E,\Delta+D)$ has Kawamata log terminal
singularities.
\end{theorem}

\begin{corollary}
\label{corollary:exceptional-criterion-1} The~singularity $(V\ni
O)$ is exceptional if $\alpha(V\ni O)>1$.
\end{corollary}

\begin{corollary}
\label{corollary:exceptional-criterion-2} If $(V\ni O)$ is
exceptional, then $(V\ni O)$ is weakly-exceptional.
\end{corollary}

Let $G$ be a~finite subgroup in $\GL_{n+1}(\mathbb{C})$, where
$n\geqslant 1$. Put
$$
\bar{G}=\phi\big(G\big)\subset\mathrm{Aut}\big(\mathbb{P}^{n}\big)\cong
\PGL_{n+1}\Big(\mathbb{C}\Big),
$$
where
$\phi\colon\GL_{n+1}(\mathbb{C})\to\mathrm{Aut}(\mathbb{P}^{n})$ is
the~natural projection. Put
$$
\mathrm{lct}\Big(\mathbb{P}^{n},\bar{G}\Big)=\mathrm{sup}\left\{\lambda\in\mathbb{Q}\ \left|%
\aligned
&\text{the~log pair}\ \left(\mathbb{P}^{n}, \lambda D\right)\ \text{has log canonical singularities}\\
&\text{for every $\bar{G}$-invariant effective $\mathbb{Q}$-divisor}\ D\sim_{\mathbb{Q}} -K_{\mathbb{P}^{n}}\\
\endaligned\right.\right\}\in\mathbb{R}.%
$$

\begin{lemma}[{\cite[Remark~3.2]{ChSh09}}]
\label{lemma:alpha-lct} Suppose  that $G$ does not contain
reflections. Then
$$
\alpha\big(V\ni O\big)=\left\{\aligned
&\mathrm{lct}\Big(\mathbb{P}^{n},\bar{G}\Big)\ \text{if}\ \mathrm{lct}\Big(\mathbb{P}^{n},\bar{G}\Big)\geqslant 1,\\
&0\ \text {if}\ \mathrm{lct}\Big(\mathbb{P}^{n},\bar{G}\Big)<1.\\
\endaligned
\right.
$$
\end{lemma}

We see that the~number $\alpha\big(V\ni O\big)$ measures how
exceptional $(V\ni O)$ is (cf. Remark~\ref{remark:Shokurov-n-2}).

\begin{example}
\label{example:Klein-Valentiner} Let $G\subset\PSL_3(\C)$. Then
$$
\mathrm{lct}\Big(\mathbb{P}^{n},\bar{G}\Big)=\left\{\aligned
&4/3\ \text{if}\ \bar{G}\cong\PSL_2\Big(\mathbb{F}_{7}\Big),\\
&2\ \text{if}\ \bar{G}\cong\A_{6},\\
\endaligned
\right.
$$
by~\cite[Examples~1.9 and~6.5]{Ch07b}.
\end{example}

For a~fixed $n\in\mathbb{Z}_{>0}$, the~number
$\mathrm{lct}(\mathbb{P}^{n},\bar{G})$ is bounded by
Theorem~\ref{theorem:primitive}, because there exist only finitely
many primitive finite subgroups in $\SL_{n+1}(\mathbb{C})$ up to
conjugation (see~\cite{Col07}).

\begin{theorem}[{\cite[Theorem~1.24]{ChSh09}}]
\label{theorem:thomas} The~inequality
$\mathrm{lct}(\mathbb{P}^{n}, {\bar{G}})\leqslant 4(n+1)$ holds
for every $n\geqslant 1$.
\end{theorem}

In fact, we expect the~following to be true (cf. \cite{Th81}).

\begin{conjecture}
\label{conjecture:Thomas} There is $\alpha\in\mathbb{R}$ such that
$\mathrm{lct}(\mathbb{P}^{n}, {\bar{G}})\leqslant \alpha$ for any
$\bar{G}\subset\mathrm{Aut}(\mathbb{P}^{n})$~and~$n\geqslant 1$.
\end{conjecture}

One can try to tackle Conjecture~\ref{conjecture:Thomas} using
the~classification of finite simple groups.

\end{document}